\documentclass{amsart}
\usepackage{amsmath}
\usepackage{amsfonts}

\newtheorem{theorem}{Theorem}
\theoremstyle{plain}

\numberwithin{equation}{section}
\input{tcilatex}

\begin{document}
\title[A Note on Bicomplex Fibonacci and Lucas Numbers]{A Note on Bicomplex
Fibonacci and Lucas Numbers}
\author{Semra KAYA NURKAN}
\author{\.{I}lkay ARSLAN G\"{U}VEN}

\begin{abstract}
In this study, we define a new type of Fibonacci and Lucas numbers which are
called bicomplex Fibonacci and bicomplex Lucas numbers. We obtain the
well-known properties e.g. D'ocagnes, Cassini, Catalan for these new types.
We also give the identities of negabicomplex Fibonacci and negabicomplex
Lucas numbers, Binet formulas and relations of them.
\end{abstract}

\maketitle

\section{\textbf{Introduction}}

Fibonacci numbers are invented by Italian mathematician Leonardo Fibonacci
when he wrote his first book Liber Abaci in 1202 contains many elementary
problems, including famous rabbit problem.

The Fibonacci numbers are the numbers of the following integer sequence;%
\begin{equation*}
1,1,2,3,5,8,13,21,34,55,89,...
\end{equation*}%
The sequence of Fibonacci numbers which is denoted by $F_{n}$ is defined as
the linear reccurence relation%
\begin{equation*}
F_{n}=F_{n-1}+F_{n-2}
\end{equation*}%
with $F_{1}=F_{2}=1$ and $n\in \mathbb{Z}$. Also as a result of this
relation we can define \ $F_{0}=0$. Fibonacci numbers are connected with the
golden ratio as; the ratio of two consecutive Fibonacci numbers approximates
the golden ratio $1,61803399...$.

Fibonacci numbers are closely related to Lucas numbers which are named after
the mathematician Francois Edouard Anatole Lucas who worked on both
Fibonacci and Lucas numbers. The integer sequence of Lucas numbers denoted
by $L_{n}$ is given by%
\begin{equation*}
2,1,3,4,7,11,18,29,47,...
\end{equation*}%
with the same reccurence relation 
\begin{equation*}
L_{n}=L_{n-1}+L_{n-2}.
\end{equation*}%
where $L_{0}=2$, $L_{1}=1$, $L_{2}=3$ and $n\in \mathbb{Z}.$

There are many works on Fibonacci and Lucas numbers in literature. The
properties, relations, results between Fibonacci and Lucas numbers can be
found in, Dunlap\cite{Dun}, Koshy\cite{Koshy}, Vajda\cite{Vajda}, Verner and
Hoggatt\cite{Verner}. Also Fibonacci and Lucas quaternions were described by
Horadam in \cite{Hor2}, then many studies related with these quaternions
were done in Akyi\u{g}it \cite{Ak}, Halici \cite{Halici}, Iyer \cite{Iyer},
Swamy \cite{Swamy}.

Corrado Segre introduced bicomplex numbers in 1892 \cite{Seg}. The bicomplex
numbers are a type of Clifford algebra, one of several possible
generalizations of the ordinary complex numbers. The complex numbers denoted
as $\mathbb{C}$ are defined by%
\begin{equation*}
\mathbb{C=}\left \{ a+bi\text{ }\mid \text{ }a,b\in \mathbb{R}\text{ and }%
i^{2}=-1\right \} .
\end{equation*}%
The bicomplex numbers are defined by%
\begin{equation*}
\mathbb{C}_{2}=\left \{ z_{1}+z_{2}j\text{ }\mid \text{ }z_{1},z_{2}\in 
\mathbb{C}\text{ and }j^{2}=-1\right \} .
\end{equation*}%
Since $z_{1}$ and $z_{2}$ are complex numbers, writing $z_{1}=a+bi$ \ and \ $%
z_{2}=c+di$ gives us another way to represent the bicomplex numbers as;%
\begin{equation*}
z_{1}+z_{2}j=a+bi+cj+dij.
\end{equation*}%
where $a,b,c,d$ $\in \mathbb{R}$. Thus the set of bicomplex numbers can be
expressed by%
\begin{equation*}
\mathbb{C}_{2}=\left \{ a+bi+cj+dij\mid \text{ }a,b,c,d\in \mathbb{R}\text{
and }i^{2}=-1,j^{2}=-1,\text{ }ij=ji=k\right \} .
\end{equation*}%
For any bicomplex numbers $x=a_{1}+b_{1}i+c_{1}j+d_{1}ij$ \ and \ $%
y=a_{2}+b_{2}i+c_{2}j+d_{2}ij,$ the addition and multiplication of these
bicomplex numbers are given respectively by%
\begin{equation}
x+y=(a_{1}+a_{2})+(b_{1}+b_{2})i+(c_{1}+c_{2})j+(d_{1}+d_{2})ij  \tag{1.1}
\end{equation}%
and%
\begin{eqnarray}
x\times y
&=&(a_{1}a_{2}-b_{1}b_{2}-c_{1}c_{2}-d_{1}d_{2})+(a_{1}b_{2}+b_{1}a_{2}-c_{1}d_{2}-d_{1}c_{2})i
\TCItag{1.2} \\
&&+(a_{1}c_{2}+c_{1}a_{2}-b_{1}d_{2}-d_{1}b_{2})j+(a_{1}d_{2}+d_{1}a_{2}+b_{1}c_{2}+c_{1}b_{2})ij.
\notag
\end{eqnarray}%
The multiplication of a bicomplex number by a real scalar $\lambda $ is
given by%
\begin{equation*}
\lambda x=\lambda a_{1}+\lambda b_{1}i+\lambda c_{1}j+\lambda d_{1}ij.
\end{equation*}%
The set $\mathbb{C}_{2}$ forms a commutative ring with addition and
multiplication and it is a real vector space with addition and scalar
multiplication \cite{Price}.

In $\mathbb{C}$, the complex conjugate of $z=a+bi$ \ is \ $\overline{z}%
=a-bi. $ In $\mathbb{C}_{2}$, for a bicomplex number $%
x=(a_{1}+b_{1}i)+(c_{1}+d_{1}i)j$, there is there different conjugations
which are;%
\begin{eqnarray}
x^{\star } &=&\left[ (a_{1}+b_{1}i)+(c_{1}+d_{1}i)j\right] ^{^{\star
}}=(a_{1}-b_{1}i)+(c_{1}-d_{1}i)j  \notag \\
x^{\circ } &=&\left[ (a_{1}+b_{1}i)+(c_{1}+d_{1}i)j\right] ^{\circ
}=(a_{1}+b_{1}i)-(c_{1}+d_{1}i)j  \TCItag{1.3} \\
x^{\dagger } &=&\left[ (a_{1}+b_{1}i)+(c_{1}+d_{1}i)j\right] ^{^{\dagger
}}=(a_{1}-b_{1}i)-(c_{1}-d_{1}i)j  \notag
\end{eqnarray}%
\cite{Roc}. Then the following equation are written in \cite{Kara};%
\begin{eqnarray}
x\times x^{\star }
&=&(a_{1}^{2}+b_{1}^{2}-c_{1}^{2}-d_{1}^{2})+2(a_{1}c_{1}+b_{1}d_{1})j 
\notag \\
x\times x^{\circ }
&=&(a_{1}^{2}-b_{1}^{2}+c_{1}^{2}-d_{1}^{2})+2(a_{1}b_{1}+c_{1}d_{1})i 
\TCItag{1.4} \\
x\times x^{\dagger }
&=&(a_{1}^{2}+b_{1}^{2}+c_{1}^{2}+d_{1}^{2})+2(a_{1}d_{1}-b_{1}c_{1})ij. 
\notag
\end{eqnarray}%
Also in \cite{Roc}, the modulus of a bicomplex number $x$ is defined by%
\begin{eqnarray}
\left \vert x\right \vert _{i} &=&\sqrt{x\times x^{\star }}  \notag \\
\left \vert x\right \vert _{j} &=&\sqrt{x\times x^{\circ }}  \TCItag{1.5} \\
\left \vert x\right \vert _{k} &=&\sqrt{x\times x^{\dagger }}  \notag \\
\left \vert x\right \vert &=&\sqrt{\func{Re}(x\times x^{\dagger })}  \notag
\end{eqnarray}%
where the names are $i-$modulus, $j-$modulus, $k-$modulus and real modulus,
respectively.

Rochon and Shapiro gave the detailed algebratic properties of bicomplex and
hyperbolic numbers in \cite{Roc}. The generalized bicomplex numbers were
defined by Karaku\c{s} and Aksoyak in \cite{Kara}, they gave some algebratic
\ properties of them and used generalized bicomplex number product to show
that $\mathbb{R}^{4}$ and $\mathbb{R}_{2}^{4}$ were Lie groups. Also
Luna-Elizarraras \textit{et al }\cite{Luna}, introduced the algebra of
bicomplex numbers and described how to define elementary functions and their
inverse functions in that algebra.

In this paper, we define bicomplex Fibonacci and bicomplex Lucas numbers by
combining bicomplex numbers and Fibonacci, Lucas numbers. We give some
identities and Binet formulas of these new numbers.

\section{\textbf{Bicomplex Fibonacci Numbers}}

\textbf{Definition 1: }The \textit{bicomplex Fibonacci} and \textit{%
bicomplex Lucas numbers} are defined respectively by%
\begin{equation}
BF_{n}=F_{n}+F_{n+1}i+F_{n+2}j+F_{n+3}k  \tag{2.1}
\end{equation}%
and%
\begin{equation}
BL_{n}=L_{n}+L_{n+1}i+L_{n+2}j+L_{n+3}k  \tag{2.2}
\end{equation}%
where $F_{n}$ is the $n^{th}$ Fibonacci number, $L_{n}$ is the $n^{th}$
Lucas number and $i,j,k$ are bicomplex units which satisfy the commutative
multiplication rules:%
\begin{eqnarray*}
i^{2} &=&-1,\text{ }j^{2}=-1,\text{ }k^{2}=1 \\
ij &=&ji=k,\text{ }jk=kj=-i,\text{ }ik=ki=-j.
\end{eqnarray*}

Starting from $n=0$, the bicomplex Fibonacci and bicomplex Lucas numbers can
be written respectively as;%
\begin{equation*}
BF_{0}=1i+1j+2k,\text{ \ }BF_{1}=1+1i+2j+3k\text{ , \ }BF_{2}=1+2i+3j+5k%
\text{,...}
\end{equation*}%
\begin{equation*}
BL_{0}=2+1i+3j+4k\text{ \ , \ }BL_{1}=1+3i+4j+7k\text{ \ , \ }%
BL_{2}=3+4i+7j+11k\text{,...}
\end{equation*}

\bigskip

Let $BF_{n}=$ $F_{n}+F_{n+1}i+F_{n+2}j+F_{n+3}k$ and $%
BF_{m}=F_{m}+F_{m+1}i+F_{m+2}j+F_{m+3}k$ be two bicomplex Fibonacci numbers.
By taking into account the equations (1.1) and (1.2), the addition,
subtraction and multiplication of these numbers are given by%
\begin{equation*}
BF_{n}\pm BF_{m}=\left( F_{n}\pm F_{m}\right) +\left( F_{n+1}\pm
F_{m+1}\right) i+\left( F_{n+2}\pm F_{m+2}\right) j+\left( F_{n+3}\pm
F_{m+3}\right) k
\end{equation*}%
and%
\begin{eqnarray*}
BF_{n}\times BF_{m} &=&\left(
F_{n}F_{m}-F_{n+1}F_{m+1}-F_{n+2}F_{m+2}+F_{n+3}F_{m+3}\right) \\
&&+\left( F_{n}F_{m+1}+F_{n+1}F_{m}-F_{n+2}F_{m+3}+F_{n+3}F_{m+2}\right) i \\
&&+\left( F_{n}F_{m+2}+F_{n+2}F_{m}-F_{n+1}F_{m+3}-F_{n+3}F_{m+1}\right) j \\
&&+\left( F_{n}F_{m+3}+F_{n+3}F_{m}+F_{n+1}F_{m+2}+F_{n+2}F_{m+1}\right) k.
\end{eqnarray*}

\textbf{Definition 2: }A bicomplex Fibonacci number can also be expressed as 
$BF_{n}=\left( F_{n}+F_{n+1}i\right) +\left( F_{n+2}+F_{n+3}i\right) j$. In
that case,there is three different \textit{conjugations} with respect to $i,$
$j$ and $k;$ for bicomplex Fibonacci numbers as follows:%
\begin{eqnarray*}
\overline{BF_{n}}^{i} &=&\left[ \left( F_{n}+F_{n+1}i\right) +\left(
F_{n+2}+F_{n+3}i\right) j\right] ^{i}=\left( F_{n}-F_{n+1}i\right) +\left(
F_{n+2}-F_{n+3}i\right) j \\
\overline{BF_{n}}^{j} &=&\left[ \left( F_{n}+F_{n+1}i\right) +\left(
F_{n+2}+F_{n+3}i\right) j\right] ^{j}=\left( F_{n}+F_{n+1}i\right) -\left(
F_{n+2}+F_{n+3}i\right) j \\
\overline{BF_{n}}^{k} &=&\left[ \left( F_{n}+F_{n+1}i\right) +\left(
F_{n+2}+F_{n+3}i\right) j\right] ^{k}=\left( F_{n}-F_{n+1}i\right) -\left(
F_{n+2}-F_{n+3}i\right) j.
\end{eqnarray*}%
By the Definition 2 and the equations (1.4) and (2.1), we can write%
\begin{eqnarray*}
BF_{n}\times \overline{BF_{n}}^{i} &=&\left( F_{2n+1}-F_{2n+7}\right)
+2\left( F_{2n+3}\right) j \\
BF_{n}\times \overline{BF_{n}}^{j} &=&\left(
F_{n}^{2}-F_{n+1}^{2}+F_{n+3}^{2}-F_{n+4}^{2}\right) +2\left(
F_{n}F_{n+1}+F_{n+2}F_{n+3}\right) i \\
BF_{n}\times \overline{BF_{n}}^{k} &=&\left( F_{2n+1}+F_{2n+7}\right)
+2\left( -1\right) ^{n+1}k.
\end{eqnarray*}

\textbf{Definition 3: }Let a bicomplex Fibonacci number be $%
BF_{n}=F_{n}+F_{n+1}i+F_{n+2}j+F_{n+3}k.$ The $i-$\textit{modulus}, $j-$%
\textit{modulus}, $k-$\textit{modulus} and \textit{real modulus }of $BF_{n}$
are given respectively as;%
\begin{eqnarray*}
\left\vert BF_{n}\right\vert _{i} &=&\sqrt{BF_{n}\times \overline{BF_{n}}^{i}%
} \\
\left\vert BF_{n}\right\vert _{j} &=&\sqrt{BF_{n}\times \overline{BF_{n}}^{j}%
} \\
\left\vert BF_{n}\right\vert _{k} &=&\sqrt{BF_{n}\times \overline{BF_{n}}^{k}%
} \\
\left\vert BF_{n}\right\vert &=&\sqrt{F_{2n+1}+F_{2n+7}}.
\end{eqnarray*}

In \cite{Nurkan}, we proved some summation equations by supposing, \ for $%
i\geq 0$%
\begin{equation}
\overset{M}{\underset{m=0}{\dsum }}\alpha _{m}F_{m+i}+\overset{M}{\underset{%
m=0}{\dsum }}\beta _{m}L_{m+i}=0  \tag{2.3}
\end{equation}%
where $\alpha _{m}$ and $\beta _{m}$ are fixed numbers.

By using the equations (2.1), (2.2), (2.3) and taking $i=0,1,2,3$, we can
write the following equation clearly;%
\begin{equation}
\overset{M}{\underset{m=0}{\dsum }}\alpha _{m}BF_{m}+\overset{M}{\underset{%
m=0}{\dsum }}\beta _{m}BL_{m}=0  \tag{2.4}
\end{equation}

Now let us give theorems which give some properties.

\begin{theorem}
Let $BF_{n}$ and $BL_{n}$ be a bicomplex Fibonacci and a bicomplex Lucas
number$,$ respectively. For $n\geq 0$ , the following relations hold:%
\begin{equation*}
\begin{array}{l}
1)\text{ \ }BF_{n}+BF_{n+1}=BF_{_{n+2}} \\ 
2)\text{ \ }BL_{n}+BL_{n+1}=BL_{_{n+2}} \\ 
3)\text{ \ }BL_{n}=BF_{n-1}+BF_{n+1} \\ 
4)\text{ \ }BL_{n}=BF_{n+2}-BF_{n-2} \\ 
5)\text{ \ }%
BF_{n}^{2}+BF_{n+1}^{2}=BF_{2n+1}+F_{2n+2}+F_{2n+5}-3iF_{2n+5}-jF_{2n+6}+3kF_{2n+4}
\\ 
6)\text{ \ }BF_{n+1}^{2}-BF_{n-1}^{2}=2BF_{2n}+F_{2n+4}+F_{2n-1}+2\left(
-F_{2n+5}i-F_{2n+4}j+F_{2n+3}k\right) \\ 
7)\text{ \ }%
BF_{n}BF_{m}+BF_{n+1}BF_{m+1}=2BF_{n+m+1}+2F_{n+m+4}-F_{n+m+1}-2F_{n+m+6}i
\\ 
\text{ \ \ \ \ \ \ \ \ \ \ \ \ \ \ \ \ \ \ \ \ \ \ \ \ \ \ \ \ \ \ \ \ \ \ \
\ \ \ }-2F_{n+m+5}j+2F_{n+m+4}k \\ 
8)\text{ \ }BF_{n}-BF_{n+1}i+BF_{n+2}j-BF_{n+3}k=-5F_{n+3}%
\end{array}%
\end{equation*}
\end{theorem}

\begin{proof}
By using the equations of (2.1), (2.2), (2.4) and taking appropriate fixed
numbers in the equation (2.4), the proofs of $1)$, $2)$, $3)$, $4)$ are
clear.

5) , 6) , 7), 8) From the definition of Fibonacci number , the bicomplex
Fibonacci number in (2.1) , the equations $F_{n}^{2}+F_{n+1}^{2}=F_{2n+1}$ , 
$F_{n+1}^{2}-F_{n-1}^{2}=F_{2n}$ $\ $and $\
F_{n}F_{m}+F_{n+1}F_{m+1}=F_{n+m+1}$ (see Vajda \cite{Vajda}), we get%
\begin{eqnarray*}
BF_{n}^{2}+BF_{n+1}^{2} &=&\left( F_{n}+F_{n+1}i+F_{n+2}j+F_{n+3}k\right)
^{2} \\
&&+\left( F_{n+1}+F_{n+2}i+F_{n+3}j+F_{n+4}k\right) ^{2} \\
&=&\left( 
\begin{array}{c}
F_{2n+3}+2\left( F_{n}F_{n+1}-F_{n+2}F_{n+3}\right) i+ \\ 
2\left( F_{n}F_{n+2}-F_{n+1}F_{n+3}\right) j+2\left(
F_{n}F_{n+3}+F_{n+1}F_{n+2}\right) k%
\end{array}%
\right) \\
&&+\left( 
\begin{array}{c}
F_{2n+5}+2\left( F_{n+1}F_{n+2}-F_{n+3}F_{n+4}\right) i \\ 
+2\left( F_{n+1}F_{n+3}-F_{n+2}F_{n+4}\right) j \\ 
+2\left( F_{n+1}F_{n+4}+F_{n+2}F_{n+3}\right) k%
\end{array}%
\right) \\
&=&F_{2n+1}+F_{2n+2}i+F_{2n+3}j+F_{2n+4}k+F_{2n+2} \\
&&+F_{2n+5}-3F_{2n+5}i-F_{2n+6}j+3F_{2n+4}k \\
&=&BF_{2n+1}+F_{2n+2}+F_{2n+5}-3iF_{2n+5}-jF_{2n+6}+3kF_{2n+4}.
\end{eqnarray*}

\begin{eqnarray*}
\text{\ }BF_{n+1}^{2}-BF_{n-1}^{2} &=&\left(
F_{n+1}+F_{n+2}i+F_{n+3}j+F_{n+4}k\right) ^{2} \\
&&-\left( F_{n-1}+F_{n}i+F_{n+1}j+F_{n+2}k\right) ^{2} \\
&=&\left( 
\begin{array}{c}
F_{2n+5}+2\left( F_{n+1}F_{n+2}-F_{n+3}F_{n+4}\right) i+ \\ 
2\left( F_{n+1}F_{n+3}-F_{n+2}F_{n+4}\right) j \\ 
+2\left( F_{n+1}F_{n+4}+F_{n+2}F_{n+3}\right)%
\end{array}%
\right) \\
&&-\left( 
\begin{array}{c}
F_{2n+1}+2\left( F_{n-1}F_{n}-F_{n+1}F_{n+2}\right) i \\ 
+2\left( F_{n-1}F_{n+1}-F_{n}F_{n+2}\right) j \\ 
+2\left( F_{n-1}F_{n+2}+F_{n}F_{n+1}\right) k%
\end{array}%
\right) \\
&=&2BF_{2n}+F_{2n+4}+F_{2n-1}+2\left( -F_{2n+5}i-F_{2n+4}j+F_{2n+3}k\right)
\end{eqnarray*}

\begin{eqnarray*}
BF_{n}BF_{m}+BF_{n+1}BF_{m+1} &=&\left(
F_{n}+F_{n+1}i+F_{n+2}j+F_{n+3}k\right) \\
&&.(F_{m}+F_{m+1}i+F_{m+2}j+F_{m+3}k) \\
&&+\left( F_{n+1}+F_{n+2}i+F_{n+3}j+F_{n+4}k\right) \\
&&.\left( F_{m+1}+F_{m+2}i+F_{m+3}j+F_{m+4}k\right) \\
&=&2\left( F_{n+m+1}+F_{n+m+2}i+F_{n+m+3}j+F_{n+m+4}k\right) \\
&&+2F_{n+m+4}-F_{n+m+1}-2F_{n+m+6}i \\
&&-2F_{n+m+5}j+2F_{n+m+4}k \\
&=&2BF_{n+m+1}+2F_{n+m+4}-F_{n+m+1}-2F_{n+m+6}i \\
&&-2F_{n+m+5}j+2F_{n+m+4}k
\end{eqnarray*}%
and%
\begin{eqnarray*}
BF_{n}-BF_{n+1}i+BF_{n+2}j-BF_{n+3}k &=&\left(
F_{n}+F_{n+1}i+F_{n+2}j+F_{n+3}k\right) - \\
&&\left( F_{n+1}+F_{n+2}i+F_{n+3}j+F_{n+4}k\right) i \\
&&+\left( F_{n+2}+F_{n+3}i+F_{n+4}j+F_{n+5}k\right) j \\
&&-\left( F_{n+3}+F_{n+4}i+F_{n+5}j+F_{n+6}k\right) k \\
&=&F_{n}+F_{n+2}-F_{n+4}-F_{n+6} \\
&=&-5F_{n+3}.
\end{eqnarray*}
\end{proof}

\begin{theorem}
For $n,m\geq 0$ the D'ocagnes identity for bicomplex Fibonacci numbers $%
BF_{n}$ $\ $and $\ BF_{m}$ is given by%
\begin{equation*}
BF_{m}BF_{n+1}-BF_{m+1}BF_{n}=\left( -1\right) ^{n}BF_{m-n}+\left( -1\right)
^{n+1}\left( 
\begin{array}{c}
F_{m-n}+F_{m-n+1}i \\ 
-\left( F_{m-n-2}+2F_{m-n+2}\right) j \\ 
+2F_{m-n-1}k%
\end{array}%
\right) .
\end{equation*}
\end{theorem}

\begin{proof}
If we decide the equation (2.1) and the D'ocagnes identity for Fibonacci
numbers $F_{m}F_{n+1}-F_{m+1}F_{n}=\left( -1\right) ^{n}F_{m-n}$ (see
Weisstein \cite{Weis} ), we obtain the following calculations as;%
\begin{eqnarray*}
BF_{m}BF_{n+1}-BF_{m+1}BF_{n} &=&\left(
F_{m}+F_{m+1}i+F_{m+2}j+F_{m+3}k\right) (F_{n+1}+F_{n+2}i \\
&&+F_{n+3}j+F_{n+4}k)-(F_{m+1}+F_{m+2}i+F_{m+3}j \\
&&+F_{m+4}k)\left( F_{n}+F_{n+1}i+F_{n+2}j+F_{n+3}k\right) \\
&=&\left[ F_{n+2}\left( F_{m}+F_{m+4}\right) -F_{m+2}\left(
F_{n}+F_{n+4}\right) \right] i \\
&&+\left[ 2\left( -1\right) ^{n}\left( F_{m-n-2}+F_{m-n+2}\right) \right] j
\\
&&+\left[ \left( -1\right) ^{n}\left( F_{m-n-2}+F_{m-n+2}\right) \right] k \\
&&+\left( -1\right) ^{n}BF_{m-n}-\left( -1\right) ^{n}BF_{m-n} \\
&=&\left( -1\right) ^{n}BF_{m-n} \\
&&+\left( -1\right) ^{n+1}\left( 
\begin{array}{c}
F_{m-n}+F_{m-n+1}i \\ 
-\left( F_{m-n-2}+2F_{m-n+2}\right) j \\ 
+2F_{m-n-1}k%
\end{array}%
\right) .
\end{eqnarray*}
\end{proof}

\begin{theorem}
If $BF_{n}$ and $BL_{n}$\ are bicomplex Fibonacci and bicomplex Lucas
numbers respectively, then for $n\geq 0$, the identities of negabicomplex
Fibonacci and negabicomplex Lucas numbers are%
\begin{equation*}
BF_{-n}=\left( -1\right) ^{n+1}BF_{n}+\left( -1\right) ^{n}L_{n}\left(
i+j+2k\right)
\end{equation*}%
and 
\begin{equation*}
BL_{-n}=\left( -1\right) ^{n}BL_{n}+\left( -1\right) ^{n+1}5F_{n}\left(
i+j+2k\right)
\end{equation*}
\end{theorem}

\begin{proof}
From the equations (2.1) and the identity of negafibonacci numbers which is $%
F_{-n}=\left( -1\right) ^{n+1}F_{n}$ \ (see Knuth \cite{Knuth} , Dunlap \cite%
{Dun}), we have%
\begin{eqnarray*}
BF_{-n} &=&F_{-n}+F_{-n+1}i+F_{-n+2}j+F_{-n+3}k \\
&=&F_{-n}+F_{-\left( n-1\right) }i+F_{-\left( n-2\right) }j+F_{-\left(
n-3\right) }k \\
&=&\left( -1\right) ^{n+1}F_{n}+\left( -1\right) ^{n}F_{n-1}i+\left(
-1\right) ^{n+1}F_{n-2}j+\left( -1\right) ^{n}F_{n-3}k \\
&=&\left( -1\right) ^{n+1}\left( F_{n}+F_{n+1}i+F_{n+2}j+F_{n+3}k\right)
-\left( -1\right) ^{n+1}F_{n+1}i \\
&&-\left( -1\right) ^{n+1}F_{n+2}j-\left( -1\right) ^{n+1}F_{n+3}k+\left(
-1\right) ^{n}F_{n-1}i \\
&&+\left( -1\right) ^{n+1}F_{n-2}j+\left( -1\right) ^{n}F_{n-3}k \\
&=&\left( -1\right) ^{n+1}BF_{n}+\left( -1\right) ^{n}\left(
F_{n+1}+F_{n-1}\right) i \\
&&+\left( -1\right) ^{n}\left( F_{n+2}-F_{n-2}\right) j+\left( -1\right)
^{n}\left( F_{n+3}+F_{n-3}\right) k
\end{eqnarray*}%
In this equation , we take into account that $F_{n-1}+F_{n+1}=L_{n}$, $\
F_{n+2}-F_{n-2}=L_{n},$ $\ F_{n+3}+F_{n-3}=2L_{n}$(see Vajda \cite{Vajda}) \
and the identity of negalucas numbers $L_{-n}=\left( -1\right) ^{n}L_{n}$
(see Knuth \cite{Knuth} ,Dunlap \cite{Dun}), thus we obtain;%
\begin{eqnarray*}
BF_{-n} &=&\left( -1\right) ^{n+1}BF_{n}+\left( -1\right) ^{n}L_{n}i+\left(
-1\right) ^{n}L_{n}j+\left( -1\right) ^{n}2L_{n} \\
&=&\left( -1\right) ^{n+1}BF_{n}+\left( -1\right) ^{n}L_{n}\left(
i+j+2k\right)
\end{eqnarray*}%
Now by using the equation (2.2) and the identity of negalucas numbers, we get%
\begin{eqnarray*}
BL_{-n} &=&L_{-n}+L_{-n+1}i+L_{-n+2}j+L_{-n+3}k \\
&=&L_{-n}+L_{-\left( n-1\right) }i+L_{-\left( n-2\right) }j+L_{-\left(
n-3\right) }k \\
&=&\left( -1\right) ^{n}L_{n}+\left( -1\right) ^{n-1}L_{n-1}i+\left(
-1\right) ^{n}L_{n-2}j+\left( -1\right) ^{n-1}L_{n-3}k \\
&=&\left( -1\right) ^{n}\left( L_{n}+L_{n+1}i+L_{n+2}j+L_{n+3}k\right)
-\left( -1\right) ^{n}L_{n+1}i \\
&&-\left( -1\right) ^{n}L_{n+2}j-\left( -1\right) ^{n}L_{n+3}k+\left(
-1\right) ^{n-1}L_{n-1}i+\left( -1\right) ^{n}L_{n-2}j \\
&&+\left( -1\right) ^{n-1}L_{n-3}k \\
&=&\left( -1\right) ^{n}BL_{n}+\left( -1\right) ^{n+1}\left( 
\begin{array}{c}
\left( L_{n+1}+L_{n-1}\right) i \\ 
+\left( L_{n+2}-L_{n-2}\right) j+\left( L_{n+3}+L_{n-3}\right) k%
\end{array}%
\right)
\end{eqnarray*}%
Here if we use the identity \ $L_{m+n}+L_{m-n}=\QDATOPD \{ . {5F_{m}F_{n}%
\text{ \ \ if \ }n\text{ \ is odd}}{L_{m}L_{n}\text{ \ \ \ \ otherwise}}$ \
\ (see Koshy, \cite{Koshy}), namely $L_{n-1}+L_{n+1}=5F_{n}$ , the
definition of dual Lucas number and the identity of negafibonacci number in
last equation, we complete the proof as;%
\begin{eqnarray*}
BL_{-n} &=&\left( -1\right) ^{n}BL_{n}+\left( -1\right) ^{n+1}5F_{n}i+\left(
-1\right) ^{n+1}5F_{n}j+\left( -1\right) ^{n+1}10F_{n}k \\
&=&\left( -1\right) ^{n}BL_{n}+\left( -1\right) ^{n+1}5F_{n}\left(
i+j+2k\right)
\end{eqnarray*}%
\bigskip
\end{proof}

The Binet formulas for Fibonacci and Lucas numbers are given by (see Koshy, 
\cite{Koshy})%
\begin{equation*}
F_{n}=\frac{\alpha ^{n}-\beta ^{n}}{\alpha -\beta }\text{ \ \ \ \ and \ \ \
\ }L_{n}=\alpha ^{n}+\beta ^{n}
\end{equation*}%
where 
\begin{equation*}
\alpha =\frac{1+\sqrt{5}}{2}\text{ \ \ \ and \ \ \ }\beta =\frac{1-\sqrt{5}}{%
2}.
\end{equation*}

\begin{theorem}
Let $BF_{n}$ and $BL_{n}$ be bicomplex Fibonacci and bicomplex Lucas
numbers, respectively. For $n\geq 0$ , the Binet formulas for these numbers
are given as;%
\begin{equation*}
BF_{n}=\dfrac{\overline{\alpha }\alpha ^{n}-\overline{\beta }\beta ^{n}}{%
\alpha -\beta }
\end{equation*}%
and%
\begin{equation*}
BL_{n}=\overline{\alpha }\alpha ^{n}+\overline{\beta }\beta ^{n}
\end{equation*}%
where $\overline{\alpha }=1+i\alpha +j\alpha ^{2}+k\alpha ^{3}$ \ and $%
\overline{\beta }=1+i\beta +j\beta ^{2}+k\beta ^{3}.$
\end{theorem}

\begin{proof}
By using the Binet formulas for Fibonacci and Lucas numbers, taking $%
\overline{\alpha }=1+i\alpha +j\alpha ^{2}+k\alpha ^{3}$ \ and $\overline{%
\beta }=1+i\beta +j\beta ^{2}+k\beta ^{3}$, we find the result clearly as;%
\begin{eqnarray*}
BF_{n} &=&F_{n}+F_{n+1}i+F_{n+2}j+F_{n+3}k \\
&=&\dfrac{\alpha ^{n}-\beta ^{n}}{\alpha -\beta }+\dfrac{\alpha ^{n+1}-\beta
^{n+1}}{\alpha -\beta }i+\dfrac{\alpha ^{n+2}-\beta ^{n+2}}{\alpha -\beta }j+%
\dfrac{\alpha ^{n+3}-\beta ^{n+3}}{\alpha -\beta } \\
&=&\frac{\alpha ^{n}\left( 1+i\alpha +j\alpha ^{2}+k\alpha ^{3}\right)
-\beta ^{n}\left( 1+i\beta +j\beta ^{2}+k\beta ^{3}\right) }{\alpha -\beta }
\\
&=&\dfrac{\overline{\alpha }\alpha ^{n}-\overline{\beta }\beta ^{n}}{\alpha
-\beta }
\end{eqnarray*}%
and%
\begin{eqnarray*}
BL_{n} &=&L_{n}+L_{n+1}i+L_{n+2}j+L_{n+3}k \\
&=&\alpha ^{n}+\beta ^{n}+(\alpha ^{n+1}+\beta ^{n+1})i+(\alpha ^{n+2}+\beta
^{n+2})j+(\alpha ^{n+3}+\beta ^{n+3})k \\
&=&\alpha ^{n}(1+i\alpha +j\alpha ^{2}+k\alpha ^{3})+\beta ^{n}(1+i\beta
+j\beta ^{2}+k\beta ^{3}) \\
&=&\overline{\alpha }\alpha ^{n}+\overline{\beta }\beta ^{n}.
\end{eqnarray*}
\end{proof}

\begin{theorem}
Let $BF_{n}$ and $BL_{n}$ be bicomplex Fibonacci and bicomplex Lucas
numbers. For $n\geq 1$ ,the Cassini identities for $BF_{n}$ and $BL_{n}$ are
given by;%
\begin{equation*}
BF_{n+1}BF_{n-1}-BF_{n}^{2}=3\left( -1\right) ^{n}\left( 2j+k\right)
\end{equation*}%
and%
\begin{equation*}
BL_{n+1}BL_{n-1}-BL_{n}^{2}=5(-1)^{n-1}\left( 2j+k\right)
\end{equation*}
\end{theorem}

\begin{proof}
From the equation (2.1), we have%
\begin{eqnarray*}
BF_{n+1}BF_{n-1}-BF_{n}^{2} &=&\left(
F_{n+1}+F_{n+2}i+F_{n+3}j+F_{n+4}k\right) (F_{n-1} \\
&&+F_{n}i+F_{n+1}j+F_{n+2}k) \\
&&-\left( F_{n}+F_{n+1}i+F_{n+2}j+F_{n+3}k\right) ^{2}.
\end{eqnarray*}%
If we use the identity $F_{m}F_{n+1}-F_{m+1}F_{n}=\left( -1\right)
^{n}F_{m-n}$ (see Weisstein \cite{Weis} ) and $F_{-n}=\left( -1\right)
^{n+1}F_{n}$ \ (see Knuth \cite{Knuth} ) in the above equation , we get%
\begin{eqnarray*}
BF_{n+1}BF_{n-1}-BF_{n}^{2} &=&\left[ 2\left( -1\right) ^{n+2}\right] 3j+%
\left[ 3\left( -1\right) ^{n+2}\right] k \\
&=&3\left( -1\right) ^{n}\left( 2j+k\right)
\end{eqnarray*}%
Similarly for bicomplex Lucas numbers we can simply get;%
\begin{eqnarray*}
BL_{n+1}BL_{n-1}-BL_{n}^{2} &=&\left(
L_{n+1}+L_{n+2}i+L_{n+3}j+L_{n+4}k\right) (L_{n-1} \\
&&+L_{n}i+L_{n+1}j+L_{n+2}k) \\
&&-\left( L_{n}+L_{n+1}i+L_{n+2}j+L_{n+3}k\right) ^{2}.
\end{eqnarray*}%
By using the identity of Lucas numbers which are $%
L_{n-1}L_{n+1}-L_{n}^{2}=5(-1)^{n-1}$ (see Koshy \cite{Koshy}) and $%
L_{n+2}=L_{n+1}+L_{n}$ (see Dunlap \cite{Dun}) in the above \ equation, we
compute the following expression;%
\begin{eqnarray*}
BL_{n+1}BL_{n-1}-BL_{n}^{2} &=&=5(-1)^{n-1}(1+\varepsilon ) \\
&=&5(-1)^{n-1}\left( 2j+k\right) .
\end{eqnarray*}
\end{proof}

\begin{theorem}
Let $BF_{n}$ be a bicomplex Fibonacci number. The Catalan's identity for $%
BF_{n}$ is given by%
\begin{equation*}
BF_{n}^{2}-BF_{n+r}BF_{n-r}=\left( -1\right) ^{n-r}\left[ 2\left(
F_{r-2}^{2}+F_{r}^{2}\right) j+\left(
F_{r+1}^{2}+F_{r-2}^{2}-F_{r}^{2}-F_{r-3}^{2}\right) k\right] .
\end{equation*}
\end{theorem}

\begin{proof}
From the equation (2.1), we have%
\begin{eqnarray*}
BF_{n}^{2}-BF_{n+r}BF_{n-r} &=&\left[ 
\begin{array}{c}
F_{n}^{2}-F_{n+1}^{2}-F_{n+2}^{2}+F_{n+3}^{2}-F_{n+r}F_{n-r}+ \\ 
F_{n+r+1}F_{n-r-1}+F_{n+r+2}F_{n-r+2}-F_{n+r+3}F_{n-r+3}%
\end{array}%
\right] \\
&&+\left[ 
\begin{array}{c}
2F_{n}F_{n+1}-2F_{n+2}F_{n+3}-F_{n+r}F_{n-r+1}- \\ 
F_{n+r+1}F_{n-r}+F_{n+r+2}F_{n-r+3}+F_{n+r+3}F_{n-r+2}%
\end{array}%
\right] i \\
&&+\left[ 
\begin{array}{c}
2F_{n}F_{n+2}-2F_{n+1}F_{n+3}-F_{n+r}F_{n-r+2}- \\ 
F_{n+r+2}F_{n-r}+F_{n+r+1}F_{n-r+3}+F_{n+r+3}F_{n-r+1}%
\end{array}%
\right] j \\
&&+\left[ 
\begin{array}{c}
2F_{n}F_{n+3}+2F_{n+1}F_{n+2}-F_{n+r}F_{n-r+3}- \\ 
F_{n+r+3}F_{n-r}-F_{n+r+1}F_{n-r+2}-F_{n+r+2}F_{n-r+1}%
\end{array}%
\right] k
\end{eqnarray*}%
Here using the Catalan's identity for Fibonacci numbers $%
F_{n}^{2}-F_{n-r}F_{n+r}=(-1)^{n-r}F_{r}^{2}$ \ (see Weisstein \cite{Weis})
, $F_{n+1}^{2}-F_{n-1}^{2}=F_{2n}$ (see Vajda \cite{Vajda}) and doing
necessary calculations, we obtain the result clearly.

Also by adding and subtracting the term $\left( -1\right) ^{n-r}BF_{r}^{2}$,
we can get the result as; 
\begin{eqnarray*}
BF_{n}^{2}-BF_{n+r}BF_{n-r} &=&\left( -1\right) ^{n-r}\left[ 2\left(
F_{r-2}^{2}+F_{r}^{2}\right) j+\left(
F_{r+1}^{2}+F_{r-2}^{2}-F_{r}^{2}-F_{r-3}^{2}\right) k\right] \\
&=&\left( -1\right) ^{n-r}BF_{r}^{2}-\left( -1\right) ^{n-r}BF_{r}^{2} \\
&&+\left( -1\right) ^{n-r}\left( 
\begin{array}{c}
2\left( F_{r-2}^{2}+F_{r}^{2}\right) j \\ 
+\left( F_{r+1}^{2}+F_{r-2}^{2}-F_{r}^{2}-F_{r-3}^{2}\right) k%
\end{array}%
\right) \\
&=&\left( -1\right) ^{n-r}\left( 
\begin{array}{c}
BF_{r}^{2}-F_{2r+3}+2F_{2r+3}i \\ 
+2\left( F_{r+2}F_{r-1}+F_{2r+1}\right) j \\ 
-\left( 
\begin{array}{c}
2F_{r}F_{r+3}+2F_{r+1}F_{r+2}+F_{r}^{2} \\ 
+F_{r-3}^{2}-F_{r+1}^{2}-F_{r-2}^{2}%
\end{array}%
\right) k%
\end{array}%
\right) .
\end{eqnarray*}
\end{proof}

\bigskip

Semra Kaya Nurkan

U\c{s}ak University

Department of Mathematics

64200, U\c{s}ak, Turkey

E-mail: semra.kaya@usak.edu.tr, semrakaya\_gs@yahoo.com

\bigskip

\.{I}lkay Arslan G\"{u}ven

Gaziantep University

Department of Mathematics

\c{S}ehitkamil, 27310, Gaziantep, Turkey

E-mail: iarslan@gantep.edu.tr, ilkayarslan81@hotmail.com


\begin{thebibliography}{99}
\bibitem{Ak} M. Akyi\u{g}it, H.H. K\"{o}sal and M. Tosun, \textit{Split
Fibonacci Quaternions}. Adv. in Appl. Clifford Algebras, 23 (2013), 535-545.

\bibitem{Clif} W. K. Clifford, \textit{Preliminary Sketch of Bi-quaternions.}
Proc. London Math. Soc., 4 (1873), 381-395.

\bibitem{Dun} R. A. Dunlap, \textit{The Golden Ratio and Fibonacci Numbers.}
World Scientific, 1997.

\bibitem{Gug} H. W. Guggenheimer, \textit{Differential Geometry. }%
McGraw-Hill Comp., 1963.

\bibitem{Halici} S. Hal\i c\i , \textit{On Fibonacci Quaternions.} Adv. in
Appl. Clifford Algebras, 22 (2012), 321-327.

\bibitem{Halici2} S. Hal\i c\i , \textit{On Complex Fibonacci Quaternions.}
Adv. in Appl. Clifford Algebras, 23 (2013), 105-112.

\bibitem{Hamil} W. R. Hamilton, \textit{Elements of Quaternions.} Longmans,
Green and Co., London, 1866.

\bibitem{Hor} A. F. Horadam, \textit{A Generalized Fibonacci Sequence.}
American Math. Monthly, 68 (1961), 455-459.

\bibitem{Hor2} A. F. Horadam, \textit{Complex Fibonacci Numbers and
Fibonacci Quaternions.} American Math. Monthly, 70 (1963), 289-291.

\bibitem{Iyer} M. R. Iyer, \textit{Some Results on \ Fibonacci Quaternions. }%
The Fibonacci Quarterly, 7(2) (1969), 201-210.

\bibitem{Kara} S. \"{O}. Karaku\c{s} and F. K. Aksoyak, \textit{Generalized
Bicomplex Numbers and Lie Groups. }Adv. in Appl. Clifford Algebras, DOI
10.1007/s00006-015-0545-x.

\bibitem{Knuth} D. Knuth, \textit{Negafibonacci Numbers and Hyperbolic
Plane. }Annual Meeting of the Math. Association of America, 2013.

\bibitem{Koshy} T. Koshy, \textit{Fibonacci and Lucas Numbers with
Applications. }A Wiley-Intersience Publication, USA, 2001.

\bibitem{Luna} M.E. Luna-Elizarraras, M. Shapiro, D.C. Struppa and A.
Vajiac, \textit{Bicomplex Numbers and Their Elementary Functions. }CUBO A
Math. Jour., 14(2) (2012), 61-80.

\bibitem{Nurkan} S. K. Nurkan and \.{I}. A. G\"{u}ven, \textit{Dual
Fibonacci Quaternions}. Adv. in Appl. Clifford Algebras, 25 (2015), 403-414.
DOI 10.1007/s00006-014-0488-7

\bibitem{Price} G. B. Price, \textit{An Intoduction to Multicomplex Spaces
and Functions. }Marcel Dekker Inc., New York, 1990.

\bibitem{Roc} D. Rochon and M. Shapiro, \textit{On Algebratic Properties of
Bicomplex and Hyperbolic numbers. }Anal. Univ. Oreda Fascicola Matematica,
11 (2004), 1-28.

\bibitem{Seg} C. Segre, \textit{Le Rappresentazioni Reali Delle Forme
Complesse e Gli Enti Iperalgebrici.} Mathematische Annalen, 40 (1892),
413-467. doi:10.1007/bf01443559

\bibitem{Swamy} M. N. Swamy, \textit{On Generalized Fibonacci Quaternions. }%
The Fibonacci Quarterly, 5 (1973), 547-550.

\bibitem{Vajda} S. Vajda, \textit{Fibonacci and Lucas Numbers and the Golden
Section. }Ellis Horwood Limited Publ., England, 1989.

\bibitem{Verner} E. Verner and Jr. Hoggatt, \textit{Fibonacci and Lucas
Numbers. }The Fibonacci Association, 1969.

\bibitem{Weis} E. W. Weisstein, \textit{Fibonacci Number. }MathWorld (online
mathematics reference work).
\end{thebibliography}
\end{document}